\documentclass[reqno,11pt,centertags,draft]{amsart}

\usepackage {amsfonts}
\usepackage[english]{babel}
\def\cc{{\mathcal{C}}}
\def\ch{{\mathcal{H}}}

\def\a{\alpha}
\def\b{\beta}
\def\l{\lambda}

\def\p{\varphi}
\def\th{\theta}
\def\d{\delta}
\def\De{\Delta}
\def\o{\omega}
\def\O{\Omega}

\def\R{{\mathbb R}}

\def\C{{\mathbb C}}
\def\N{{\mathbb N}}
\def\D{{\mathbb D}}
\def\T{{\mathbb T}}
\def\e{\varepsilon}
\def\g{\gamma}
\def\G{\Gamma}
\def\ep{\varepsilon}
\def\z{\zeta}

\def\ovl{\overline}

\def\dist{{\rm dist}}

\def\bs{~\hfill\rule{7pt}{7pt}}

\newtheorem{theorem}{Theorem}

\newtheorem*{definition}{Definition}

\newtheorem{proposition}[theorem]{Proposition}

\begin{document}

\title[Blaschke-type conditions]
{Blaschke-type conditions for analytic functions in the unit disk:
inverse problems and local analogs}
\author[S. Favorov and L. Golinskii]{S. Favorov and L. Golinskii}

\address{Mathematical School, Kharkov National University, Swobody sq.4,
Kharkov, 61077 Ukraine}
\email{Sergey.Ju.Favorov@univer.kharkov.ua}

\address{Mathematics Division, Institute for Low Temperature Physics and
Engineering, 47 Lenin ave., Kharkov 61103, Ukraine}
\email{leonid.golinskii@gmail.com}

\date{\today}


\begin{abstract}
We continue the study of analytic functions in the unit disk of
finite order with arbitrary set of singular points on the unit
circle, introduced in \cite{FG}. The main focus here is made upon
the inverse problem: the existence of a function from this class
with a given singular set and zero set subject to certain
Blaschke-type condition. We also discuss the local analog of the
main result from \cite{FG} similar to the standard local Blaschke
condition for analytic and bounded functions in the unit disk.
\end{abstract}

\maketitle

\section{Introduction}

In early 20s V. V. Golubev completed his treatise entitled ``The
study on the theory of singular points of single valued analytic
functions'' which was published as a series of papers in 1924--1927
in ``Uchenye zapiski gosudarstvennogo saratovskogo universiteta''
(the edition hardly known nowadays to the former Soviet Union
experts and completely unavailable to the Western audience).  A book
\cite{Gol} which came out in 1961 in Russian, after the author's
death, contains this treatise as the second part. The author made an
attempt to develop the theory of functions in the unit disk
$\D=\{|z|<1\}$ by using the similarity between such functions and
the entire functions which arises if one views the unit circle
$\T=\{|t|=1\}$ as an analog of the unique singular point at infinity
for entire functions. Golubev managed to prove a number of results
which were latter attributed to either the ``mathematical folklore''
or other authors. For instance, he was by far the first one to show
that for a function $f$ of finite order at most $\rho$ in $\D$
\begin{equation}\label{finor}
|f(z)|\le h(1-|z|), \quad z\in\D, \qquad
h(x)=\exp\left(\frac1{x}\right)^{\rho}, \quad \rho>0,
\end{equation}
its zero set $Z(f)=\{z_n\}$ (each zero $z_n$ is counted according to
its multiplicity) obeys a Blaschke-type condition
\begin{equation}\label{blcond}
\sum_{z_n\in Z(f)}(1-|z_n|)^{\rho+1+\ep}<\infty, \qquad \forall\ep>0
\end{equation}
(see, e.g., \cite[Chapter II, \S 1]{Gol}). This result was extended
to a wide class of weight functions $h$ by F.~Shamoyan \cite{Sh2}.

Let $E$ be an arbitrary closed subset of the unit circle. In
\cite[p. 113]{AC76} P.~Ahern and D. Clark defined a {\it type of
$E$} by
\begin{equation}\label{degspar}
\b(E):=\sup\{\b\in\R:\ |E_x|=O(x^\b), \ x\to 0\},
\end{equation}
where
$$ E_x:=\{t\in\T: \ \hat d(t,E)<x\}, \qquad x>0, $$
an $x$-neighborhood of $E$, $\hat d(t_1,t_2)$ is the length of the
shortest arc with endpoints $t_1, t_2$, $|E_x|$ its normalized
Lebesgue measure ($|\T|=1$). This characteristic was suggested by
the authors in their study of inner functions with derivatives in
certain functional classes. It is clear that $0\le\b(E)\le1$,
$\b(E)=1$ for finite (and close in a sense to finite) sets, and
$\b(E)=0$ for all sets $E$ of positive Lebesgue measure (see Section
\ref{local} for more results and examples concerning $\b(E)$).

In a recent paper \cite{FG} we introduced a class $\ch(\rho,E)$ of
analytic functions in $\D$ of finite order at most $\rho$ having an
arbitrary closed set $E\subset\T$ as the set of singular points,
i.e.,
\begin{equation}\label{finore}
|f(z)|\le \exp\left(\frac{C}{d^\rho(z,E)}\right), \quad z\in\D,
\qquad d(z,E)={\rm dist}\,(z,E).
\end{equation}
Being unaware of \cite{AC76}, \footnote{The authors thank I. E.
Chyzhykov for drawing our attention to this paper.} we defined the
value $\b(E)$ in an equivalent way (cf. Proposition \ref{integr}
below) and proved in \cite[Theorem 3]{FG}, that the Blaschke-type
condition for such functions is
\begin{equation}\label{blconde}
\sum_{z_n\in Z(f)}(1-|z_n|)\,d^{(\rho-\b(E)+\ep)_+}(z_n,E)<\infty,
\qquad \forall\ep>0,
\end{equation}
where, as usual, $a_+:=\max(a,0)$. So \eqref{blconde} turns into
\eqref{blcond} for $E=\T$.

One of the goals of the present paper is to study the inverse
problems: given a closed set $E\subset\T$ and a discrete set, i.e.,
having no limit points in $\D$, $Z\subset\D$ subject to
Blaschke-type condition \eqref{blconde}, whether there exists a
function $f\in\ch(\rho,E)$ such that $Z=Z(f)$. The standard and well
respected method here is to use various canonical products
(Weierstrass, Nevanlinna, Golubev, Djrbashian, Tsuji). The problem
is to pick one with the smallest possible growth near $E$, once the
convergence exponent of zeros is known. The first example of such
type is due to Golubev (see Proposition \ref{golub} below). Some
related results for the case $E=\T$ can be found in \cite{Dj, G, G1,
L1, mamo, N, Sh2, T}.

Our contribution in that direction is the following

\begin{theorem}\label{maininv}
Let $E=\bar E\subset\T$, $\rho>0$, and $Z=\{z_n\}_{n\ge1}\subset\D$
satisfy
\begin{equation}\label{blconin}
K:=\sum_{n\ge1}(1-|z_n|)d^\rho(z_n,E)<\infty.
\end{equation}
Then there exists an analytic function $f\in\ch(\rho+1,E)$ such that
$Z(f)=Z$. \end{theorem}
\smallskip

Next, we turn to local versions of the Blaschke conditions, which
are less known and appreciated by experts. In the notation
$B(a,r)=\{z\in\C:\ |z-a|<r\}$, let $f$ be an analytic and bounded
function in the lune
$$ L_\tau(t)=\D\cap B(t,\tau), \qquad t\in\T. $$
Then the portion of its zero set inside any interior lune $L(t,\d)$,
$\d<\tau$ is subject to the Blaschke condition. Indeed, denote by
$F=F_\tau$ a conformal mapping from $L_\tau(t)$ to $\D$,
$\Phi=F^{(-1)}$, so $g=f(\Phi)$ is a bounded, analytic function in
$\D$, hence
$$ \sum_{w_n\in Z(g)}(1-|w_n|)=\sum_{z_n\in
Z(f)}(1-|F(z_n)|)<\infty, $$ and the more so
$$ \sum_{z_n\in Z(f)\cap L_\d}(1-|F(z_n)|)<\infty $$
for any interior lune $L_\d\subset L_\tau$. Due to the well-known
properties of conformal mappings (see also the explicit expression
for $F$ in Section \ref{local})
$$ 1-|F(z)|\asymp 1-|z|, \qquad z\in L_\d, $$
(as usual, $a\asymp b$ means $c\le a/b\le C$), $\{z_n\in Z(f)\cap
L_\d\}$ is the Blaschke sequence, as claimed.

In Section \ref{local} of the paper we develop certain extensions of
such local Blascke condition for the class $\ch(\rho,E)$ (see
Theorem \ref{localfg}).

\section{Canonical products in the disk and inverse
problems}\label{inverse}

\subsection{Prime factors of Weierstrass and Nevanlinna}
We begin with the bounds for Weierstrass and Nevanlinna prime
factors.

\begin{definition}
A Weierstrass prime factor of order $p$ is
\begin{equation}\label{primew}
W(z,p):=(1-z)\exp\left(\sum_{k=1}^p \frac{z^k}{k}\right), \quad
p=1,2,\ldots, \quad W(z,0)=1-z.
\end{equation}
\end{definition}
The following bounds for $W$ are well known.
\begin{proposition}\label{weiboun}
We have
\begin{enumerate}
    \item $|\log W(z,p)|\le \frac32\,|z|^{p+1}$,
    $p=0,1,\ldots$ for $|z|<\frac13$;
    \item $\log|W(z,p)|\le A_p|z|^{p}$,
    $p=1,2,\ldots$ for $|z|\ge\frac13$, $A_p=3e(2+\log p)$;
    \item $|1-W(z,p)|\le |z|^{p+1}$,
    $p=0,1,\ldots$ for $|z|\le1$.
\end{enumerate}
\end{proposition}

Following \cite{Nev}, we define a {\it Nevanlinna prime factor of
order $p$} by
\begin{equation}\label{primen}
N_p(w,\o)=
\frac{W\left(\frac{w}{\o},p\right)}{W\left(\frac{w}{\ovl\o},p\right)}\,.
\end{equation}
The result below is actually proved in \cite[Lemma 3.2]{Gov}.
\begin{proposition}\label{nevboun}
Let $\o=|\o|e^{i\tau}$, $\C_+=\{z:\Im z>0\}$ an upper half plane.
Then
\begin{enumerate}
    \item $|\log N_p(w,\o)|\le
    3|\sin\tau|\,\left|\frac{w}{\o}\right|^{p+1}$
    for $|w/\o|<\frac13$;
    \item $\log|N_p(w,\o)|\le C_p|\sin\tau|\,\left|\frac{w}{\o}\right|^{p}$,
    for $|w/\o|\ge\frac13$ and $w,\o\in\C_+$.
    \end{enumerate}
\end{proposition}
\begin{proof} (1). We have
$$ \log
N_p(w,\o)=\sum_{k=p+1}^\infty\left(\frac1{\bar\o^k}-\frac1{\o^k}\right)\,\frac{w^k}{k}
=2i\sum_{k=p+1}^\infty \frac{\sin k\tau}{|\o|^k}\,\frac{w^k}{k}, $$
and since $|\sin k\tau|\le k|\sin\tau|$ then
$$ |\log N_p(w,\o)|\le
2|\sin\tau|\,\left|\frac{w}{\o}\right|^{p+1}\sum_{j\ge0}\frac1{3^j}=
3|\sin\tau|\,\left|\frac{w}{\o}\right|^{p+1}. $$

(2). Now
$$ |N_p(w,\o)|=\left|\frac{1-w/\o}{1-w/\ovl\o}\right|\,
   \exp\left\{-2\Im
   \sum_{k=1}^p\Im\left(\frac1{\o^k}\right)\,\frac{w^k}{k}\right\} $$
and since $w,\o\in\C_+$, then $|1-w/\o|\le |1-w/\ovl\o|$ so
$$
\log |N_p(w,\o)|\le
2|\sin\tau|\sum_{k=1}^p\,\left|\frac{w}{\o}\right|^{k} \le
C_p|\sin\tau|\,\left|\frac{w}{\o}\right|^{p}.
$$
\end{proof}

Note also that
$$
N_p(w,\o)-1=\frac{W\left(\frac{w}{\o},p\right)-W\left(\frac{w}{\ovl\o},p\right)}
{W\left(\frac{w}{\ovl\o},p\right)}, $$ and by Proposition
\ref{weiboun}, (3)
$$ |N_p(w,\o)-1|\le
\frac{2\left|\frac{w}{\o}\right|^{p+1}}{1-\left|\frac{w}{\o}\right|^{p+1}}<\frac14,
\quad \left|\frac{w}{\o}\right|<\frac13. $$ The standard inequality
for logarithms
$$ \frac12\,|\z|\le|\log(1+\z)|\le \frac32\,|\z|, \qquad |\z|<\frac12 $$
now gives
\begin{equation}\label{nminus1}
|N_p(w,\o)-1|\le 2|\log N_p(w,\o)|\le
6|\sin\tau|\,\left|\frac{w}{\o}\right|^{p+1}, \quad
\left|\frac{w}{\o}\right|<\frac13. \end{equation}

\subsection{Inverse problems}
Canonical products appear to be an appropriate tool for solving
inverse problems. The statement below illustrates this fact nicely
(the idea is due to Golubev, cf. \cite[Chapter I, \S 3]{Gol}).
\begin{proposition}\label{golub}
Let $Z=\{z_n\}\in\D$ satisfy
\begin{equation}\label{golubinv}
K:=\sum_{n\ge1} d^{\rho+1}(z_n,E)<\infty, \qquad \rho>0,
\end{equation}
$E=\ovl E\subset\T$. Then there exists a function
$f\in\ch(\rho+1,E)$ such that $Z(f)=Z$.
\end{proposition}
\begin{proof} Denote by $e_n\in E$ one of the closest points to $z_n$, so $d(z_n,E)=|z_n-e_n|$.
Put
$$ G(z):=\prod_{n\ge1} W(v_n,p), \quad v_n=\frac{z_n-e_n}{z-e_n}, $$
$p$ satisfies %is the smallest positive integer such that $p\ge\rho$.
\begin{equation}\label{prho}
\rho\le p<\rho+1, \qquad p\in\N=\{1,2,\ldots\}.
\end{equation}
We call $G$ a {\it Golubev canonical product}. Since $v_n(z)\to 0$
for each $z\in\D$, then $|v_n(z)|\le 1$ for $n\ge n_0(z)$. By
Proposition \ref{weiboun}, (3) and \eqref{prho}
$$ \sum_{n\ge n_0} |W(v_n,p)-1|\le \sum_{n\ge n_0}
|v_n(z)|^{p+1} \le \sum_{n\ge n_0}
|v_n(z)|^{\rho+1}\le\frac{K}{d^{\rho+1}(z,E)}\,, $$
 so the product $G$ converges absolutely and uniformly in $\D$. Write
$$ G(z)=\Pi_1(z)\cdot\Pi_2(z), $$
where
$$ \Pi_1(z)=\prod_{n:|v_n|>1}W(v_n,p), \qquad \Pi_2(z)=\prod_{n:|v_n|\le
1}W(v_n,p). $$

Since $\log|W(v_n,p)|\le |W(v_n,p)-1|$, then as above
$$ \log|\Pi_2(z)|\le \sum_{|v_n|\le 1}|W(v_n,p)-1|\le
\frac{K}{d^{\rho+1}(z,E)}. $$ As for the first factor, by
Proposition \ref{weiboun}, (2), and \eqref{prho}
$$ \log|\Pi_1(z)|\le C_p\sum_{n:|v_n|>1} |v_n|^{p}
\le C_p\sum_{n:|v_n|>1} |v_n|^{\rho+1}\le\frac{C_p
K}{d^{\rho+1}(z,E)},
$$  that leads to the final bound
$$ |G(z)|\le \exp\left\{\frac{B_\rho K}{d^{\rho+1}(z,E)}\right\}. $$
It is clear that the zero set $Z(f)=Z$. \end{proof}

\smallskip

{\it Proof of Theorem \ref{maininv}}. Note that for $E=\T$ we have
exactly Proposition \ref{golub}, so with no loss of generality
assume that $E\not=\T$, and moreover,
\begin{equation}\label{eneravt}
-1\notin E, \qquad d(-1,E)=2\d>0, \quad \d=\d(E).
\end{equation}
We can also assume that $|\l+1|\ge\d$ for all $\l\in Z$, since for
the part of $Z$ with $|\l+1|<\d$ the standard Blaschke condition
holds by \eqref{blconin} and \eqref{eneravt}, so the corresponding
Blaschke product represents this part of $Z$.

\if{It is not hard to see that $Z$ can be split into a disjoint
union of sets $Z=\cup Z_n$, $Z_k\cap Z_j=\emptyset$ for $j\not=k$ so
that
\begin{equation}\label{splitz}
\sum_{z_k\in Z_n} (1-|z_k|)|e_n-z_k|^\rho<\infty, \qquad
n=1,2,\ldots, \qquad e_n\in E.
\end{equation}
In fact, take $e_n\in E$ to be the closest point to $z_n$, and put
\begin{equation}\label{consg}
G_n:=\{z\in Z: |e_n-z|\le 2d(z,E)\}.
\end{equation}
Clearly, $z_n\in G_n$, $Z=\cup G_n$, so there exists a sequence of
disjoint sets $Z_n\subset G_n$ with the union $Z$. \eqref{splitz}
follows directly from the construction of $G_n$ and \eqref{blconin}.
If
$$ K_n:=\sum_{z_k\in Z_n} (1-|z_k|)d^\rho(z_k,E), \qquad K=\sum_n K_n, $$
then by \eqref{consg}
\begin{equation}\label{kn}
\sum_{z_k\in Z_n} (1-|z_k|)|e_n-z_k|^\rho\le 2^\rho\,K_n.
\end{equation}

We label the zeros in $Z_n$ by $Z_n=\{z_{n,j}\}$,
$j=1,2,\ldots,\a_n$ with $\a_n\le\infty$.}\fi

Take $t=e^{i\th}\in E$, $|\th|<\pi$, and consider an auxiliary
linear-fractional mapping $g_t:\D\rightarrow\C_+$
\begin{equation}\label{linfrac}
g_t(\l):=ie^{i\frac{\th}2}\,\frac{1+\l}{t-\l}\,, \qquad
g_t(t)=\infty, \quad g_t(-1)=0.
\end{equation}
It is clear that the following bounds hold for $g$
\begin{equation}\label{boung}
\begin{split}
|g_t(\l)| &\le \frac2{|t-\l|}\,, \qquad \qquad \l\in\D, \\
|g_t(\l)| &\ge \frac{\d}{|t-\l|}\ge\frac{\d}2\,, \qquad |\l+1|\ge\d.
\end{split}
\end{equation}
Similarly, for the imaginary part
$$ \Im g_t(\l)=\cos\frac{\th}2\,\frac{1-|\l|^2}{|t-\l|^2} $$
one has
\begin{equation}\label{bounimg}
\cos\frac{\th}2\,\frac{1-|\l|}{|t-\l|^2}\le\Im g_t(\l)\le
2\,\frac{1-|\l|}{|t-\l|^2}\,, \quad \l\in\D.
\end{equation}
As above, $p$ satisfies \eqref{prho}.

Let $z,\l\in\D$. If $|g_t(\l)|>3|g_t(z)|$, then by Proposition
\ref{nevboun}, (1), and \newline $\rho+1\le p+1$
\begin{equation}\label{nevupper1}
\begin{split}
\left|\log N_p(g_t(z), g_t(\l))\right| &\le 3|\sin\arg
g_t(\l)|\,\left|\frac{g_t(z)}{g_t(\l)}\right|^{p+1} \\ &=\frac{3\Im
g_t(\l)}{|g_t(\l)|^{\rho+2}}\,\left|\frac{g_t(z)}{g_t(\l)}\right|^{p+1-\rho-1}\,
|g_t(z)|^{\rho+1} \\ &\le \frac{3\Im g_t(\l)}{|g_t(\l)|^{\rho+2}}\,
|g_t(z)|^{\rho+1}.
\end{split}
\end{equation}
If $|g_t(\l)|\le 3|g_t(z)|$, then by Proposition \ref{nevboun}, (2),
and $p<\rho+1$
\begin{equation}\label{nevupper2}
\begin{split}
\log|N_p(g_t(z), g_t(\l))| &\le C_\rho|\sin\arg
g_t(\l)|\,\left|\frac{g_t(z)}{g_t(\l)}\right|^{p} \\
&=\frac{C_\rho\Im
g_t(\l)}{|g_t(\l)|^{\rho+2}}\,\left|\frac{g_t(z)}{g_t(\l)}\right|^{\rho+1-p}\,
|g_t(z)|^{\rho+1} \\ &\le C_\rho\,2^{\rho+1-p}\,\frac{\Im
g_t(\l)}{|g_t(\l)|^{\rho+2}}\, |g_t(z)|^{\rho+1}.
\end{split}
\end{equation}

In view of \eqref{boung}--\eqref{bounimg} and \eqref{nminus1} we
finally have for $z\in\D$, $|\l+1|\ge\d$
\begin{equation}\label{nevupper3}
\log|N_p(g_t(z), g_t(\l))|\le
C_{\rho,\d}\,\frac{(1-|\l|)|t-\l|^\rho}{|t-z|^{\rho+1}}\,,
\end{equation}
and
\begin{equation}\label{nevupper4}
|N_p(g_t(z), g_t(\l))-1|\le 2|\log N_p(g_t(z), g_t(\l))|\le
C_{\rho,\d}\,\frac{(1-|\l|)|t-\l|^\rho}{|t-z|^{\rho+1}}\,,
\end{equation}
if in addition $|g_t(\l)|>3|g_t(z)|$.
\smallskip

Let, as above, $e_n\in E$ be one of the closest points to $z_n$.
Consider a Nevanlinna canonical product (cf. \cite[Chapter 1]{Gov})
\begin{equation}\label{brick}
f(z):=\prod_{n=1}^{\infty}\,N_p(g_n(z),g_n(z_n)), \qquad
g_n=g_{e_n}.
\end{equation}
By \eqref{nevupper3} its partial products $f_m$ are uniformly
bounded in the disk
\begin{equation}\label{unifboun1}
\begin{split}
|f_m(z)| &=\prod_{n=1}^{m}\,|N_p(g_n(z),g_n(z_n))| \\
&\le\exp\left\{C_{\rho,\d}\sum_{n=1}^m\frac{(1-|z_n|)|e_n-z_n|^\rho}{|e_n-z|^{\rho+1}}\right\}
\le\exp\left\{\frac{C_{\rho,\d}K}{d^{\rho+1}(z,E)}\right\}\,,
\end{split}
\end{equation}
so the family $\{f_m\}$ is locally bounded in $\D$, and by Montel's
Theorem it is normal there. We will show that it converges on a
curve in $\D$ to a function, which is not identically zero.

Put $\eta=\frac14\,\d^2$ and let $z=x\in I=(-1,-1+\eta)$. Then for
$t\in E$
$$
|g_t(x)|=\frac{|1+x|}{|t-x|}<\frac{\eta}{|t+1|-|1+x|}<\frac{\eta}{2\d-\eta}<\frac{\d}6.
$$
On the other hand, as we have assumed, $|z_k+1|\ge\d$ for all
$z_k\in Z$, so by \eqref{boung} $|g_t(z_k)|>\frac12\d$ and hence
$$ |g_t(z_k)|>3|g_t(x)|, \qquad x\in I, \quad z_k\in Z. $$
Therefore \eqref{nevupper4} applies, and
$$ \sum_{n\ge1}|N_p(g_n(x),g_n(z_n))-1|<\infty, $$
so \eqref{brick} converges absolutely and uniformly on $I$ to a
function, which is not identically zero, as claimed. Hence, by
Vitali's Theorem the product $f$ converges uniformly on compact
subsets of $\D$ and as in \eqref{unifboun1}
\begin{equation}\label{bounbrick}
\log|f(z)|\le \frac{C_{\rho,\d}\,K}{d^{\rho+1}(z,E)}\,.
\end{equation}
It is clear that $Z(f)=Z$, and the proof is complete. \bs

\medskip

Let $E$ be a closed set on $\T$, $f$ a function of finite order in
$\D$ with the singular set $E$. We define the {\it order of growth
of $f$ near $E$} by
   $$
 \rho_E[f]=\inf\{\rho\in\R:\log|f(z)|<C_f d^{-\rho}(z,E)\}.
  $$
Let $Z=\{z_n\}$ be a set of points in $\D$ counted according to
their multiplicity. We define the {\it convergence exponent with
respect to $E$} by
  $$
 \rho_E[Z]=\inf\{\rho\in\R:\sum_{z_n\in
 Z}(1-|z_n|)d^\rho(z_n,E)<\infty\}.
  $$
In these terms the main result of \cite{FG} states that
$$ \rho_E[f]\le\rho \Rightarrow\rho_E[Z(f)]\le (\rho-\b(E))_+\,, $$
$\b(E)$ is the type of $E$ \eqref{degspar}. On the other hand,
Theorem \ref{maininv} is equivalent to the following: for an
arbitrary $Z$ with $\rho_E[Z]\le\rho$ there is a function $f$ such
that $\rho_E[f]\le\rho+1$ and $Z(f)=Z$. The two results can be
combined in one two sided inequality
\begin{equation}\label{twosidein}
\rho_E[Z]+\b(E)\le\inf_{f:Z(f)=Z}\rho_E[f]\le\rho_E[Z]+1.
\end{equation}
The example below shows that the right bound in \eqref{twosidein} is
attained for arbitrary $E$.

For $0<r<1$, $0\le\th<2\pi$ put
$$ \Box(r,\th):=\{z\in\D: r\le|z|\le\frac{1+r}2, \quad |\arg
z-\th|\le\pi(1-r)\}. $$ A theorem of Linden \cite[Theorem II]{L1}
claims that for a function of finite order $\rho_\T[f]$ in $\D$
\begin{equation}\label{lind1}
\nu_f(r,\th):=\sharp\{z\in
Z(f)\cap\Box(r,\th)\}=O\left(\frac1{(1-r)^{(\rho_\T[f]+\ep)}}\right),
\quad r\to 1
\end{equation}
for any $\th$ and $\ep>0$. Here $\sharp(A)$ is a number of points in
a set $A$.

Let $Z\subset\D$ and
$$ \nu(r,\th,Z):=\sharp\{z\in Z\cap\Box(r,\th)\}. $$
By Linden's theorem any lower bound for $\nu(r,\th,Z)$ for at least
one value of $\th$ implies the lower bound for the order
$\rho_\T[f]$ of each $f$ with $Z(f)\supset Z$.

{\bf Example}. Assume with no loss of generality that $1\in E$, and
put
\begin{equation}\label{zone}
Z=\{z_n\}: \quad
z_n=r_n=1-\left(\frac1{n+1}\right)^{\frac1{\rho+1}}, \quad \rho>0.
\end{equation}
Given $0<r<1$ write $k=k(r)$ so that
$$ r_{k-1}<r\le
r_k<r_{k+1}<\ldots<r_{k+l_k-1}\le\frac{1+r}2<r_{k+l_k}. $$ It is
clear that
\begin{equation*}
\begin{split}
\frac1{(1-r)^{\rho+1}}-1 &\le k<\frac1{(1-r)^{\rho+1}}\,, \\
(2^{\rho+1}-1)k &\le l_k<(2^{\rho+1}-1)k+2^{\rho+1}+1.
\end{split}
\end{equation*}
Hence $\nu(r,0,Z)=l_k\asymp (1-r)^{-\rho-1}$, and for $f$ with
$Z(f)\supset Z$
$$ \nu_f(r,0)\ge \nu(r,0,Z)\ge \frac{C}{(1-r)^{\rho+1}}\,. $$
By Linden's theorem $\rho_\T[f]\ge\rho+1$, and finally
$\rho_E[f]\ge\rho_\T[f]\ge\rho+1$.

On the other hand, $d(z_n,E)=1-r_n$, so $\rho_E[Z]=\rho$, and we
come to an equality
$$ \inf_{f:Z(f)=Z}\rho_E[f]=\rho_E[Z]+1=\rho+1, $$
as claimed.

\smallskip

For the sets $E$ close to finite ones \eqref{twosidein} turns into
equality
\begin{equation}\label{righteq}
\inf_{f:Z(f)=Z}\rho_E[f]=\rho_E[Z]+1,
\end{equation}
for the sets $E$ of positive Lebesgue measure we have
\begin{equation}\label{poslebme}
\rho_E[Z]\le\inf_{f:Z(f)=Z}\rho_E[f]\le\rho_E[Z]+1.
\end{equation}
In the case $E=\T$ Linden \cite{L1} proved that in fact
$$ \rho_{\T}[Z]\le\rho_{\T}[f]\le\rho_{\T}[Z]+1 $$
for all functions $f$ with $Z(f)=Z$. And even more so, for any $R$
in $[\rho_{\T}[Z], \rho_{\T}[Z]+1]$ there is such function $f$ that
$\rho_{\T}[f]=R$.

\section{Local Blaschke-type conditions}\label{local}

\subsection{Type of closed set}
Let $E=\bar E\subset\T$,
$$ E_x:=\{t\in\T: \ \hat d(t,E)<x\}, \qquad x>0 $$
an $x$-neighborhood of $E$, $\hat d(t_1,t_2)$ is the length of the
shortest arc with endpoints $t_1, t_2$, $|E_x|$ its normalized
Lebesgue measure. It is easily checked that
\begin{equation}\label{neib}
E_x=\bigcup\G: \ \G {\rm \ an \  open \
 arc}, \quad |\G|=x, \quad
\G\bigcap E\not=\emptyset. \end{equation} Being an open set,
$E_x=\cup_{j=1}^\o I_j$ is a disjoint union of open arcs. It is easy
from \eqref{neib} that $|I_j|\ge x$, and so
$$\o<\infty, \qquad E\bigcap I_j\not=\emptyset $$
for all $j=1,2,\ldots,\o$.

Some elementary properties of $|E_x|$ are listed in \cite[Section
2]{FG}. There is another one, which is a bit more subtle
\begin{equation}\label{sublin}
|E_{rx}|\le r |E_x|, \quad r\ge1.
\end{equation}
Indeed, it is verified directly when $E$ is a finite set. In
general, as a compact set, $E$ contains a finite $\ep$-net $\hat E$
$\forall \ep>0$, i.e., $\hat E\subset E\subset \hat E_\ep$. Hence
$$ |E_{rx}|\le |\hat E_{rx+\ep}|\le r|\hat E_{x+\ep/r}|\le
r|E_{x+\ep/r}| $$ and tend $\ep$ to $0$, as claimed.

There is a simple way to compute $|E_x|$ in terms of the
complimentary arcs of $E$. Let
$$
\T\setminus E=\bigcup_j\,\g_j, \qquad|\g_j|\downarrow0.
$$ Then
\begin{equation}\label{type1}
|E_x|=\sum_{j=N+1}^\infty |\g_j|+\frac{2N}{\pi}\arcsin\frac{x}2+|E|,
\end{equation}
where $N=N(x)$ is taken from
\begin{equation}\label{type2}
|\g_{N+1}|\le\frac2{\pi}\,{\arcsin\frac{x}2}<|\g_N|.
\end{equation}

The number $\b(E)$ \eqref{degspar}, introduced by P. Ahern and D.
Clark, is called the type of $E$. It is clear that $0\le\b(E)\le1$,
and
\begin{equation}\label{monota}
E_1\subset E_2\Rightarrow \b(E_1)\ge \b(E_2).
\end{equation}

Based on \eqref{type1}--\eqref{type2}, it is easy to manufacture
examples of countable sets $E$ with prescribed values of
$\b(E)\in[0,1]$. For instance, $\b(E)=1$ for
$$
E:=\left\{e^{i\p_n}: \p_n=\sum_{k=n}^\infty \frac1{2^k} \right\}\cup
\{1\},
$$ $\b(E)=1-\frac1{\g}$, $\g>1$, for
$$
E=\left\{e^{i\p_n}: \p_n=c\sum_{k=n}^\infty \frac1{k^\g}, \quad
\g>1\right\}\cup\{1\}, \quad c^{-1}=\sum_{k=1}^\infty
\frac1{k^\g}\,,
$$ $\b(E)=0$ for
$$
E=\left\{e^{i\p_n}: \p_n=c\sum_{k=n}^\infty \frac1{k\log^2
k}\right\}\cup\{1\}, \quad c^{-1}=\sum_{k=2}^\infty \frac1{k\log^2
k}\,.
$$
For the generalized Cantor set $\cc_\o$ (the standard Cantor set is
$\cc_{1/3}$)
$$ \b(\cc_\o)=1-d(\o), \qquad d(\o)=\frac{\log 2}{\log 2-\log(1-\o)} $$
is the Hausdorff dimension of $\cc_\o$.

The relation between $\b(E)$ and the value $I(\b,E)$ introduced in
\cite{FG} is straightforward.

\begin{proposition}\label{integr}
Given $E=\ovl E\subset\T$, let
$$ I(\b,E):=\int_0^\pi \frac{|E_x|}{x^{\b+1}}\,dx\le\infty, \qquad \b\in\R. $$
Then $\b<\b(E)$ \ $(\b>\b(E))$ implies $I(\b,E)<\infty$ \
$(I(\b,E)=\infty)$. \end{proposition}
\begin{proof} If $\b<\b(E)$, then by definition
$|E_x|=O(x^{\b+\ep})$ for some $\ep>0$, so the integral converges.

Conversely, assume that $I(\b,E)<\infty$. Then $\b<1$, and as by
\eqref{sublin} $|E_u|u^{-1}$ is a decreasing function in $u$, we
have for all small enough $y>0$
$$ 1\ge\int_0^y \frac{|E_x|}{x^{\b+1}}\,dx\ge
\frac{|E_y|}{y}\,\int_0^y\frac{dx}{x^\b}=\frac{|E_y|}{y}\,\frac{y^{1-\b}}{1-\b}\,,
$$
so $|E_y|\le(1-\b)y^\b$, and hence $\b\le\b(E)$. The proof is
complete. \end{proof}

Now, \cite[Theorem 3]{FG} states that
$$ f\in\ch(\rho,E)\Rightarrow\rho_E[Z(f)]\le (\rho-\b(E))_+. $$

The following property of $\b(E)$ proves helpful.
\begin{proposition}\label{confinv}
Let $E=\bar E\subset\G$, $\G$ is an open arc of $\T$. Let $F$ be an
analytic function in the domain $\O\supset\G$, $|F|=1$ on $\G$, and
$|F'|\asymp 1$ in $\O$. Then $\b(F(E))=\b(E)$. \end{proposition}
\begin{proof} Under the assumptions we have $|E_x|\asymp |F(E_x)|\asymp |(F(E))_x|$, as needed.
\end{proof}

\subsection{Local analogs}
A number $\b(\z)=\lim_{\d\to0}\b(E\cap\overline{B(t,\d)})$ will be
called a {\it local type of the closed set $E\subset\D$ at the point
$t\in E$}. For $t\in\D\setminus E$ we put $\b(t)=+\infty$. It is
clear that $\b(t)\ge\b(E)$ for each $t\in E$.

We are aimed at proving a local version of \cite[Theorem 3]{FG}, as
stated below.

\begin{theorem}\label{localfg}
Let  $E$ be a closed subset of $\T$, and  $t\in E$. Assume that $f$
is an analytic function in the lune $L_\tau(t)$, $0<\tau<1$, and
$$ \log|f(z)|\le\frac{C}{d^\rho(z,E)}\,, \qquad z\in L_\tau(t). $$
Then for any $\e>0$ there exists $\d=\d(\ep,t,E)>0$, $\d<\tau$ such
that inside the interior lune $L_\d(t)$
\begin{equation}\label{localbl}
\sum_{z\in Z(f)\cap L_\d(t)}(1-|z|)
\dist^{(\rho-\b(t)+\ep)_+}(z,E)<\infty.
\end{equation}
\end{theorem}
\begin{proof} With no loss of generality we put $t=1$ and for any $\xi\le\tau$
denote
$$ L_\xi:=L_\xi(1), \qquad E^\xi:=E\cap\overline L_\xi, $$
a portion of $E$ in the closure of $L_\xi$, $L_\xi\subseteq L_\tau$.

Let us show first that
\begin{equation}\label{dist1}
d(z,E)\asymp d(z,E^\xi), \qquad z\in L_\xi.
\end{equation}
Indeed, write $L_\xi=L_\xi^+\cup L_\xi^-$ with
$L_\xi^+=L_\xi\cap\{\Im z\ge0\}$, $L_\xi^-=L_\xi\cap\{\Im z<0\}$,
and $e^{\pm i\a}$, $\a>0$, for the vertices of $L_\xi$. If
$e^{i\a}\in E$ then for all $z\in L_\xi^+$
$$ d(z,E^\xi)\le d(z,e^{i\a})\le d(z,E\setminus E^\xi), $$
so $d(z,E)=\min\{d(z,E^\xi),d(z,E\setminus E^\xi)\}=d(z,E^\xi)$. If
$e^{i\a}\notin E$ then $d_+:=d(\bar L_\xi^+,E\setminus E^\xi)>0$, so
$$ d(z,E)\ge \frac{d_+}2\,d(z,E^\xi). $$
The argument for $L_\xi^-$ is the same, and we are done. So,
\begin{equation}\label{localbou}
\log|f(z)|\le\frac{C}{d^\rho(z,E^\xi)}\,, \qquad z\in L_\xi, \quad
\forall \xi\le\tau.
\end{equation}

We will distinguish two situations.

\noindent
 1. Let $\b(1)=0$, and so $\b(E^\xi)=0$ for all $\xi<\tau$. Let
 $F_\tau$ be a conformal mapping of $L_\tau$ onto $\D$,
 $\Phi_\tau=F_\tau^{(-1)}$ the inverse mapping. The explicit
 expression for $F_\tau$ is available
  $$
F_\tau(z)=\frac{i\l(z)+1}{\l(z)+i},\qquad
\l(z)=\left(\frac{(z-e^{-i\a})(1-e^{i\a})}{(z-e^{i\a})(1-e^{-i\a})}\right)^{1/\kappa},
 $$
where $\pi\kappa$ is the internal angle of $L_\tau$ at the vertices
$e^{\pm i\a}$. Note that $\kappa\in(1/2,\,2/3)$ whenever $\tau<1$.
We see that $|F'_\tau(z)|\le C$, $z\in\overline L_\tau$, hence
\begin{equation}\label{dist2}
 d(F_\tau(z_1),F_\tau(z_2))\le Cd(z_1,z_2), \qquad
z_1,z_2\in\overline L_\tau. \end{equation}

The function $g=f(\Phi)$ is analytic in $\D$, and by
\eqref{localbou} with $\xi=\tau$ and \eqref{dist2}
$$ \log|g(w)|\le \frac{C}{d^\rho(\Phi(w),E^\tau)}\le
\frac{C}{d^\rho(w,F_\tau(E^\tau))}\,. $$
 By the global result and $Z(g)=F_\tau(Z(f))$
$$ \sum_{z_n\in
Z(f)}(1-|F_\tau(z_n)|)d^{\rho+\ep}(F_\tau(z_n),F_\tau(E^\tau))<\infty,
$$
and the more so,
$$
\sum_{z_n\in Z(f)\cap
L_\d}(1-|F_\tau(z_n)|)d^{\rho+\ep}(F_\tau(z_n),F_\tau(E^\tau))<\infty
$$
for any interior lune $L_\d$. Since
$$ |F_\tau'(z)|\asymp1, \qquad \frac{1-|F_\tau(z)|}{1-|z|}\asymp1,
\quad z\in L_\d, $$ we have
$$
\sum_{z_n\in Z(f)\cap L_\d}(1-|z_n|)d^{\rho+\ep}(z_n,E)<
\sum_{z_n\in Z(f)\cap L_\d}(1-|z_n|)d^{\rho+\ep}(z_n,E^\tau)<\infty,
$$
which is \eqref{localbl} with $\b(1)=0$.

\noindent
 2. The case $\b(1)>0$ is more delicate, as we have to take care of
 the exponent in \eqref{localbl}. By the definition now $|E^\xi|=0$
 for all small enough $\xi>0$. It is not hard to see that we can
 choose two lunes $L_\eta\supset L_\d$ with $\d<\eta<\tau$ such that
 \begin{enumerate}
 \item $\b(E^\d)>\b(1)-\ep/2$.
 \item The vertices of both $L_\eta$
 and $L_\d$ are not in $E$, and $E^\eta=E^\d$.
\end{enumerate}

Let $F_\eta$ be a conformal mapping of $L_\eta$ (not $L_\tau$!) onto
$\D$, $\Phi_\eta=F_\eta^{(-1)}$ the inverse mapping. The function
$g=f(\Phi)$ is analytic in $\D$, and as above
$$ \log|g(w)|\le \frac{C}{d^\rho(\Phi(w),E^\d)}\le
\frac{C}{d^\rho(w,F_\eta(E^\d))}\,. $$
 By the global result and $Z(g)=F_\eta(Z(f))$
$$ \sum_{z_n\in
Z(f)}(1-|F_\eta(z_n)|)d^{(\rho-\b(\hat
E^\d)+\ep/2)_+}(F_\eta(z_n),F_\eta(E^\d))<\infty, \quad \hat
E^\d=F_\eta(E^\d).
$$
Next, Proposition \ref{confinv} applies, so $\b(\hat
E^\d)=\b(E^\d)>\b(1)-\ep/2$, and
$$ \sum_{z_n\in
Z(f)}(1-|F_\eta(z_n)|)d^{(\rho-\b(1)+\ep)_+}(F_\eta(z_n),F_\eta(E^\d))<\infty.
$$
As in the first case, now
$$
\sum_{z_n\in Z(f)\cap
L_\d}(1-|z_n|)d^{(\rho-\b(1)+\ep)_+}(z_n,E)<\infty,
$$
as claimed. The proof is complete.
\end{proof}

\begin{theorem}\label{disjoint}
Let $E=\cup_{j=1}^n E_j$ be a disjoint union of closed sets, $f$ an
analytic function in the unit disk which satisfies
\begin{equation}\label{disjboun}
\log|f(z)|\le \frac{C_f}{\prod_{k=1}^n d^{\rho_k}(z,E_k)}\,, \quad
\rho_k>0. \end{equation}
 Then $\forall\ep>0$
\begin{equation}\label{joinbl}
\sum_{z_n\in Z(f)}(1-|z_n|)\min_k d^{q_k}(z_n,E_k)<\infty, \quad
q_k:=(\rho_k-\b(E_k)+\ep)_+.
\end{equation}
\end{theorem}
\begin{proof}For $t\in E_j$ we
apply Theorem \ref{localfg} with $\rho=\rho_j$, $\b(t)\ge\b(E_j)$
and sufficiently small $\tau$, and come to
$$ \sum_{z_n\in Z(f)\cap L_\d(t)}(1-|z_n|) d^{q_j}(z_n,E)<\infty. $$
Assume that $\d<\frac{\De}4$, $\De:=\min_{l,s}(E_l,E_s)>0$, so for
$\l\in L_\d(t)$ and $k\not=j$
$$ d(\l,E_j)<\d<\frac{\De}4, \quad
d(\l,E_k)=|\l-e_k|=|\l-t+t-e_k|>\De-\d>\frac34\,\De, $$ i.e.,
$d(\l,E_k)>3d(\l,E_j)$. Hence
$$ \sum_{z_n\in Z(f)\cap L_\d(t)}(1-|z_n|) d^{q_j}(z_n,E_j)<\infty. $$
When $t\notin E$, for small enough $\d$
$$ \sum_{z_n\in Z(f)\cap L_\d(t)}(1-|z_n|)<\infty, $$
which is the standard local Blaschke condition. It remains only to
choose a finite covering $\T\subset\cup_{p=1}^m B(t_p,\d_p)$ and
take into account that \newline $Z(f)\setminus \cup_{p=1}^m
B(t_p,\d_p)$ is a finite set.
\end{proof}

For the finite $E$ the result was proved in \cite[Theorem 0.3]{bgk}.

\end{document}